\documentclass[tbtags,reqno]{amsart}
\usepackage{amsmath,amsthm,amsopn,amsfonts,amssymb,epsfig,enumerate}
\usepackage[all]{xy}

\edef\savecatcodeat{\the\catcode`@} \catcode`\@=11

\def\tb@ifSpecChars#1#2{#1}
\def\tb@ifNoSpecChars#1#2{#2}

\def\tableau{%
  \bgroup
  \@ifstar{\let\Tif\tb@ifNoSpecChars\tb@tableauB}
          {\let\Tif\tb@ifSpecChars\tb@tableauB}}

\def\tb@tableauB{
  \@ifnextchar[{\tb@tableauC}{\tb@tableauC[]}}

\def\tb@tableauC[#1]{\hbox\bgroup%
    \let\\=\cr
    \def\bl{\global\let\tbcellF\tb@cellNF}%
    \def\tf{\global\let\tbcellF\tb@cellH}
    \dimen2=\ht\strutbox \advance\dimen2 by\dp\strutbox%
    \ifx\baselinestretch\undefined\relax%
    \else%
       \dimen0=100sp \dimen0=\baselinestretch\dimen0%
       \dimen2=100\dimen2 \divide\dimen2 by\dimen0%
    \fi%
    \let\tpos\tb@vcenter
    \tb@initYoung
    \tb@options#1\eoo
    \let\arrow\tb@arrow%
    \dimen0=\Tscale\dimen2%
    \dimen1=\dimen0 \advance\dimen1 by \tb@fframe%
    \lineskip=0pt\baselineskip=0pt
    \def\tb@nothing{}%
    \def\endcellno{$\rss\egroup\bss\egroup}
    \def\endcell{\endcellno\kern-\dimen0}
    \def\begincell{\vbox to\dimen0\bgroup\vss\hbox to\dimen0\bgroup\hss$}%
    \let\overlay\tb@overlay%
    \let\fl\tb@fl%
    \let\lss\hss\let\rss\hss\let\tss\vss\let\bss\vss
    \def\mkcell##1{
        \let\tbcellF\tb@cellD
        \def\tb@cellarg{##1}
        \ifx\tb@cellarg\tb@nothing\let\tb@cellarg\tb@cellE\fi%
        \begincell\tb@cellarg\endcellno
        \tbcellF}
    \let\savecellF\tbcellF
     \Tif{\catcode`,=4\catcode`|=\active}{}\tb@tableauD}%

\let\tb@savetableauD\tableauD
{
    \catcode`|=\active \catcode`*=\active \catcode`~=\active%
    \catcode`@=\active
\gdef\tableauD#1{%
  \Tif{
    \mathcode`|="8000 \mathcode`*="8000%
    \mathcode`~="8000 \mathcode`@="8000%
    \def@{\bullet}%
    \let|\cr
    \let*\tf
    \let~\sk
  }{}%
  \tpos{\tabskip=0pt\halign{&\mkcell{##}\cr#1\crcr}}%
  \global\let\tbcellF\savecellF
  \egroup
  \egroup}
}
\let\tb@tableauD\tableauD
\let\tableauD\tb@savetableauD
\let\tb@savetableauD\undefined

\def\tb@options#1{\ifx#1\eoo\relax\else\tb@option#1\expandafter\tb@options\fi}

\def\tb@option#1{%
  \if#1t\let\tpos\tb@vtop\fi
  \if#1c\let\tpos\tb@vcenter\fi
  \if#1b\let\tpos\vbox\fi
  \if#1F\tb@initFerrers\fi
  \if#1Y\tb@initYoung\fi
  \if#1s\tb@initSmall\fi
  \if#1m\tb@initMedium\fi
  \if#1l\tb@initLarge\fi
  \if#1p\tb@initPartition\fi
  \if#1a\tb@initArrow\fi
}

\def\tb@vcenter#1{\ifmmode\vcenter{#1}\else$\vcenter{#1}$\fi}

\def\tb@vtop#1{\hbox{\raise\ht\strutbox\hbox{\lower\dimen0\vtop{#1}}}}

\def\tb@initPartition{\def\Tscale{.3}}
\def\tb@initSmall{\def\Tscale{1}}
\def\tb@initMedium{\def\Tscale{2}}
\def\tb@initLarge{\def\Tscale{3}}

\def\tb@initArrow{\dimen2=1.25em}

\def\tb@initYoung{%
  \def\tb@cellE{}
  \let\tb@cellD\tb@cellN
  \def\sk{\global\let\tbcellF\tb@cellNF}}
\def\tb@initFerrers{%
  \def\tb@cellE{\bullet}
  \let\tb@cellD\tb@cellNF
  \def\sk{\bullet}}

\tb@initMedium

\def\tb@sframe#1{%
  \vbox to0pt{
    \vss
    \hbox to0pt{%
      \hss
      \vbox to\dimen1{
        \hrule depth #1 height0pt
        \vss
        \hbox to\dimen1{
          \vrule width #1 height\dimen1
          \hss
          \vrule width #1
          }%
        \vss
        \hrule height #1 depth 0in
        }%
      \kern-\tb@hframe
      }%
    \kern-\tb@hframe}}

\def\tb@hframe{.2pt}\def\tb@fframe{.4pt}\def\tb@bframe{1.2pt}
\def\tb@cellH{\tb@sframe{\tb@bframe}}       
\def\tb@cellNF{}                            
\def\tb@cellN{\tb@sframe{\tb@fframe}}       
\let\tbcellF\tb@cellN                       

\def\tb@rpad{1pt}
\def\tb@lpad{1pt}
\def\tb@tpad{1.8pt}
\def\tb@bpad{1.8pt}

\def\tb@overlay{\endcell\@ifnextchar[{\tb@overlaya}{\begincell}}
\def\tb@overlaya[#1]{\vbox to\dimen0\bgroup%
  \tb@overlayoptions#1\eoo%
  \tss\hbox to\dimen0\bgroup\lss}
\def\tb@overlayoptions#1{\ifx#1\eoo\relax\else\tb@overlayoption#1\expandafter\tb@overlayoptions\fi}

\def\tb@overlayoption#1{
  \if#1t\def\tss{\vskip\tb@tpad}\let\bss\vss\fi
  \if#1c\let\tss\vss\let\bss\vss\fi
  \if#1b\def\bss{\vskip\tb@bpad}\let\tss\vss\fi
  \if#1l\def\lss{\hskip\tb@lpad}\let\rss\hss\fi
  \if#1m\let\lss\hss\let\rss\hss\fi
  \if#1r\def\rss{\hskip\tb@rpad}\let\lss\hss\fi
}

\def\tb@fl{\endcell\begincell\vrule depth 0pt width \dimen0 height \dimen0 \endcell\begincell}

\@ifundefined{diagram}{}{
\def\tb@arrowpad{.5}

\newoptcommand{\tb@arrow}{\@ne}[2]{%
  \endcell
   \begingroup%
   \let\dg@getnodesize\tb@getnodesize
   \dg@USERSIZE=#1\relax%
   \ifnum\dg@USERSIZE<\@ne \dg@USERSIZE=\@ne \fi%
   \dg@parse{#2}%
   \dg@label{\tb@draw{#1}{#2}}}

\def\tb@getnodesize#1#2#3#4#5{\dimen3=\tb@arrowpad\dimen2 #4=\dimen3 #5=\dimen3\relax}
\def\tb@getnodesize#1#2#3#4#5{\ifnum#2=0\ifnum#3=0\tb@getnodesizetail{#4}{#5}\else\tb@getnodesizehead{#4}{#5}\fi\else\tb@getnodesizehead{#4}{#5}\fi}
\def\tb@getnodesizetail#1#2{\dimen3=.5\dimen2 #1=\dimen3 #2=\dimen3}
\def\tb@getnodesizehead#1#2{\dimen3=.5\dimen2 #1=\dimen3 #2=\dimen3}

\def\tb@draw#1#2#3#4{%
        \dg@X=0\dg@Y=0\dg@XGRID=1\dg@YGRID=1\unitlength=.001\dimen0%
        \dg@LBLOFF=\dgLABELOFFSET \divide\dg@LBLOFF\unitlength%
        \dg@drawcalc
        \begincell
        \let\lams@arrow\tb@lams@arrow
        \begin{picture}(0,0)\begingroup\dg@draw{#1}{#2}{#3}{#4}\end{picture}%
        \endcell
        \endgroup
        \begincell}
}

\def\tb@lams@arrow#1#2{%
 \lams@firstx\z@\lams@firsty\z@
 \lams@lastx#1\relax\lams@lasty#2\relax
 \lams@center\z@
 \N@false\E@false\H@false\V@false
 \ifdim\lams@lastx>\z@\E@true\fi
 \ifdim\lams@lastx=\z@\V@true\fi
 \ifdim\lams@lasty>\z@\N@true\fi
 \ifdim\lams@lasty=\z@\H@true\fi
 \NESW@false
 \ifN@\ifE@\NESW@true\fi\else\ifE@\else\NESW@true\fi\fi
 \ifH@\else\ifV@\else
  \lams@slope
  \ifnum\lams@tani>\lams@tanii
   \lams@ht\ten@\p@\lams@wd\ten@\p@
   \multiply\lams@wd\lams@tanii\divide\lams@wd\lams@tani
  \else
   \lams@wd\ten@\p@\lams@ht\ten@\p@
   \divide\lams@ht\lams@tanii\multiply\lams@ht\lams@tani
  \fi
 \fi\fi
 \ifH@  \lams@harrow
 \else\ifV@ \lams@varrow
 \else \lams@darrow
 \fi\fi
}

\catcode`\@=\savecatcodeat
\let\savecatcodeat\undefined

\numberwithin{equation}{section}

\makeatletter
\renewcommand{\subsubsection}{\@startsection
{subsubsection} {3} {0mm} {\baselineskip} {-0.5\baselineskip} {\normalfont\normalsize\bfseries}} \makeatother

\newtheorem{theorem}{Theorem}
\newtheorem{lemma}[theorem]{Lemma}
\newtheorem{proposition}[theorem]{Proposition}
\newtheorem{example}[theorem]{Example}

\newtheorem{corollary}[theorem]{Corollary}

\newtheorem{definition}[theorem]{Definition}

\newtheorem{property}[theorem]{Property}

\newtheorem{remark}[theorem]{Remark}

\def\la{{\lambda}}

\def\cal L{{\mathcal L}}

\def\aa{\alpha}

\def \core {{\mathfrak c}}
\def \kbnd { \mathfrak p}

\def\gg {\gamma}

\def\aa {\alpha}

\def \part {\vdash}

\begin{document}

\title[A recursion formula for $k$-Schur functions]
{A recursion formula for $k$-Schur functions}

\author{Daniel Bravo}
\address{Department of Mathematics and Computer Science,
Wesleyan University,
Science Tower 655, 265 Church St.,
Middletown, CT 06459-0128 USA}
\email{dbravovivall@wesleyan.edu}

\author{Luc Lapointe}
\thanks{L. L. was partially supported by the Anillo Ecuaciones Asociadas a Reticulados financed by the World
Bank through the Programa Bicentenario de Ciencia y Tecnolog\'{\i}a, NSF grant DMS-0652641,
and by the Programa Reticulados y Ecuaciones
of the Universidad de Talca}
\address{Instituto de Matem\'atica y F\'{\i}sica, Universidad de Talca,
Casilla 747, Talca, Chile} \email{lapointe@inst-mat.utalca.cl}


\begin{abstract}   The Bernstein operators allow to build recursively the
Schur functions.
We present a recursion formula for $k$-Schur functions at $t=1$
based on combinatorial operators that generalize
the Bernstein operators.
The recursion leads immediately to a combinatorial
interpretation for the expansion coefficients of $k$-Schur functions at $t=1$
in terms of homogeneous
symmetric functions.
\end{abstract}


\maketitle

\section{Introduction}

The study of Macdonald polynomials led to the discovery of
symmetric functions, $s_\lambda^{(k)}(t)$, indexed by partitions
whose first part is no larger than a fixed integer $k\geq 1$, and depending on
a parameter $t$.
Experimentation suggested that when $t=1$, the functions
$s_{\lambda}^{(k)}:=s_\lambda^{(k)}(1)$ play the
fundamental combinatorial role of the Schur basis in
the symmetric function subspace $\Lambda^{k}=\mathbb Z[h_1,\ldots,h_k]$;
that is, they satisfy properties generalizing
classical properties of Schur functions such
as Pieri and Littlewood-Richardson rules.  The study of the
$s_\lambda^{(k)}$ led to several characterizations
\cite{[LLM],[LM1],[LMproofs]} (conjecturally
equivalent) and to the proof of many of these combinatorial conjectures.
We thus generically call the functions $s_\lambda^{(k)}$ $k$-Schur functions
(at $t=1$), but in
this article consider only the definition presented in \cite{[LMproofs]}.

The Bernstein operators $B_n = \sum_{i \geq 0} (-1)^i h_{n+i}  \,
e_{i}^{\perp}$ (see Section~\ref{sectbern} for the relevant definitions)
allow to
build the Schur functions recursively.  That is,
if $\lambda=(\lambda_1,\dots,\lambda_\ell)$ is a partition, we have
\begin{equation} \label{formulaBclass}
B_{\lambda_1} B_{\lambda_2} \cdots B_{\lambda_\ell} = s_{\lambda} \, ,
\end{equation}
where $s_{\lambda}$ is the Schur function indexed by the partition $\lambda$.

We present in this article a recursion for $k$-Schur functions that
generalizes this recursion.
It is based on
a combinatorial generalization of the Bernstein operators,
\begin{equation} \label{formulaB}
B_{\lambda_1}^{(k)} B_{\lambda_2}^{(k)} \cdots B_{\lambda_\ell}^{(k)} = s_{\lambda}^{(k)} \, ,
\end{equation}
where $\lambda=(\lambda_1,\dots,\lambda_\ell)$ is a partition such that
$\lambda_1 \leq k$.  In this case, the operators $B_i^{(k)}$ are only
defined on certain subspaces of $\Lambda^k$, preventing for instance
the study of the commutation relations that they could have satisfied.
The term
combinatorial is used to emphasize the fact that the operators $B_i^{(k)}$
are defined through their action on certain $k$-Schur functions,
an action which is combinatorially much in the spirit of the action
of the usual $B_i$'s on Schur functions.

Formula \eqref{formulaB} leads immediately to a
combinatorial
interpretation for the expansion coefficients of $k$-Schur functions
in terms of homogeneous symmetric functions.  This interpretation is
particularly relevant given that no Jacobi-Trudi type
determinantal formula has yet been obtained for $k$-Schur functions.

It was shown in \cite{[Lam]} that the $k$-Schur functions form the Schubert basis for
the homology of the affine (loop) Grassmannian of $GL_{k+1}$.  This implies
that the $k$-Littlewood-Richardson coefficients $c_{\mu \nu}^{\lambda, k}
\in \mathbb N$
appearing in
\begin{equation}
s_{\mu}^{(k)} \, s_{\nu}^{(k)} = \sum_{\lambda}  c_{\mu \nu}^{\lambda, k} \,
s_{\lambda}^{(k)}
\end{equation}
are the structure constants for the homology of the affine (loop)
Grassmannian of $GL_{k+1}$.  Furthermore, these
coefficients are known \cite{[LMdual]}
to contain as a subset the Gromow-Witten invariants of the quantum
cohomology ring of the Grassmannian (or equivalently the fusion rules for
the Wess-Zumino-Witten conformal field theories in case A).
Finding combinatorial interpretations for these quantities is still an open
problem.  Given that the Bernstein operators have been used  to derive the
Littlewood-Richardson rule (see \cite{[Ho]}), we hope that the recursion formula for
$k$-Schur functions presented in this article will provide some of
the insight needed for finding
such combinatorial interpretations.

\section{Definitions  }
\subsection{Basic definitions}  Most of the definitions in this
subsection are taken from \cite{[M]}.
A partition $\lambda=(\lambda_1,\dots,\lambda_m)$ is a non-increasing sequence of positive integers. The degree of
$\lambda$ is $|\lambda|=\lambda_1 +\cdots +\lambda_m$ and the length $\ell(\lambda)$ is the number of parts $m$. Each
partition $\lambda$ has an associated Ferrers diagram with $\lambda_i$ lattice squares in the $i^{th}$ row, from the
bottom to top.  For example,
\begin{equation}
\lambda\,=\,(4,2,1)\,=\, {\tiny{\tableau*[scY]{\cr  & \cr & & & }}}
\end{equation}
Given a partition $\lambda$, its conjugate $\lambda'$ is
the diagram obtained by
reflecting  $\lambda$ about the main
diagonal.  A partition $\lambda$ is ``{\it $k$-bounded}'' if $\lambda_1 \leq k$.  Any lattice square in the Ferrers
diagram is called a cell, where the cell $(i,j)$ is in the $i$th row and $j$th column of the diagram. We say that
$\lambda \subseteq \mu$ when $\lambda_i \leq \mu_i$ for all $i$.  The dominance order $\unrhd$ on partitions is defined
by $\lambda\unrhd\mu$ when $\lambda_1+\cdots+\lambda_i\geq \mu_1+\cdots+\mu_i$ for all $i$, and $|\lambda|=|\mu|$.

A skew diagram $\mu/\lambda$, for any partition $\mu$
containing the partition $\lambda$, is the diagram obtained by
deleting the cells of $\lambda$ from $\mu$.
The thick frames below represent (5,3,2,1)/(4,2).
\begin{equation*}
{\tiny{\tableau*[scY]{ \tf \cr \tf & \tf \cr  & &
\tf \cr & & & & \tf }}}
\end{equation*}
The degree of a skew diagram $\mu/\lambda$ is $|\mu|-|\lambda|$.
We say that the skew diagram $\mu/\lambda$ of degree $\ell$
is a horizontal (resp. vertical) $\ell$-strip if it never has two
cells in the same column (resp. row).

A cell $(i,j)$ of a partition $\gg$ with $(i+1,j+1)\not\in \gg$ is called ``{\it extremal }''.  A ``{\it removable}''
corner of partition $\gg$ is a cell $(i,j)\in \gg$ with $(i ,j+1),(i+1,j)\not\in \gg$ and an ``{\it  addable}'' corner
of $\gg$ is a square $(i,j)\not\in \gg$ with $(i ,j-1),(i-1,j)\in \gg$ (note
that $(1,\gg_1+1),(\ell(\gg)+1,1)$ are also considered to be addable corners).
All removable corners are extremal. In the
figure below we have labeled all addable corners with $a$, labeled extremal cells $e$, and framed the removable
corners.
\begin{equation*}
{\tiny{\tableau[scY]{\bl,\bl {{a}} , \bl  ,\bl | \bl,e, \tf e, \bl {{a}} ,\bl | \bl,,{{e}},\tf e,\bl {{}},\bl  , \bl ,
\bl |
 \bl,,,{{e}},\bl {{a}},\bl ,\bl |
\bl,,, {{e}}, e, \tf e,\bl {{a}}}}}
\end{equation*}
The hook-length of a cell $c=(i,j)$ in a partition $\lambda$ is $\lambda_i-j+\lambda_j'-i+1$. That is, the number of
cells in $\lambda$ to the right of $c$ plus the number of cells in $\lambda$ above $c$ plus one.  If, as above,
$\lambda=(5,3,3,2)$, the hook-length of the cell $(1,2)$ is 7.  We will say
that a cell is {\it $k$-bounded} if its hook-length is not larger than $k$.

Recall that a {\it ``$k+1$-core} is a partition that does not contain any $k+1$-hooks (see \cite{[JK]} for a discussion
of cores and residues). An example of a 6-core (with the hook-length of each cell indicated) is:
\begin{equation*}
{\tiny{\tableau[scY]{1|4,2,1|5,3,2|7,5,4,1}}}
\end{equation*}
The ``{\it $k+1$-residue}'' of square  $(i,j)$ is $j-i \mod k+1$. That is, the integer in this square when squares are
periodically labeled with $0,1,\ldots,k$, where zeros lie on the main diagonal.  Here are the 5-residues associated to
$(6,4,3,1,1,1)$
$$
{\tiny{\tableau[scY]{0|1|2|3,4,0|4,0,1,2| 0,1,2,3,4,0}}}
$$

We will need the following basic result on cores \cite{[GKS],[LMcore]}.
\begin{proposition} \label{propo4.1}
Let $\gg$ be a $k+1$-core.
\begin{itemize}
\item[{\rm (1)}]
Let $c$ and $c'$ be extremal cells of $\gg$ with the same $k+1$-residue ($c'$ weakly north-west of $c$).
\begin{itemize}
\item[{(a)}] If $c$ is at the end of its row, then so is $c'$.
\item[{(b)}] If $c$ has a cell above it, then so does $c'$.
\end{itemize}
\item[{\rm (2)}] Let $c$ and $c'$ be extremal cells of
$\gg$ with the same $k+1$-residue ($c'$ weakly south-east of $c$).
\begin{itemize}
\item[{(a)}] If  $c$ is at the top of its column, then so is $c'$.
\item[{(b)}] If $c$ has a cell to its right, then so does $c'$.
\end{itemize}
\item[{\rm (3)}] A $k+1$-core $\gg$ never has both a removable corner and an addable corner
of the same $k+1$-residue.
\end{itemize}
\end{proposition}

\def \CC {{\mathcal C}}
\def \CP {{\mathcal P}}
\def \CCk {\CC_{k+1}}

\subsection{Bijection: $k+1$-cores and $k$-bounded partitions}
Let $\CC^{k+1}$ and $\CP^k$ respectively denote the collections of $k+1$ cores and $k$-bounded partitions. There is a
bijective correspondence between $k+1$-cores and  $k$-bounded partitions that was defined in \cite{[LMcore]} by the
map:
$$
\kbnd: \CC^{k+1}\to \CP^k\quad \text{where}\quad \kbnd\left(\gg\right) = (\lambda_1,\ldots,\lambda_\ell) \,,
$$
with $\lambda_i$ denoting the number of cells with a $k$-bounded hook in row $i$ of $\gamma$.  Note that the number of
$k$-bounded hooks in $\gamma$ is $|\lambda|$.  The inverse map
$\core=\kbnd^{-1}$
relies on constructing a certain skew diagram $\gg/\rho$
from $\lambda$, and setting $\core(\lambda)=\gamma$. These special skew diagrams are defined:

\begin{definition} \label{def5}
For $\lambda\in\mathcal P^k$, the ``$k$-skew diagram of $\lambda$" is the diagram $\lambda/^k=\gg/\rho$ where

(i) the number of cells in
row $i$ of $\lambda/^k$ is $\lambda_i$ for $i=1,\ldots,\ell(\lambda)$

(ii) no cell of $\lambda/^k$ has hook-length exceeding $k$

(iii) all cells of $\rho$ have hook-lengths exceeding $k+1$ (when considered
in $\lambda$).

\noindent
\end{definition}

A convenient algorithm for constructing the diagram of $\lambda/^k$ is given by successively attaching a row of length
$\lambda_i$ to the bottom of $(\lambda_1,\dots,\lambda_{i-1})/^k$ in the leftmost position so that no hook-lengths
exceeding $k$ are created.

\begin{example} \label{exskew}
Given $\lambda =(4,3,2,2,1,1)$ and $k=4$,
\begin{equation*}
\lambda = {\tiny{\tableau*[scY]{  \cr  \cr  & \cr & \cr & & \cr & & & }}} \quad \Longleftrightarrow \quad \lambda/^4 =
{\tiny{\tableau*[scY]{ \cr \cr  & \cr \bl &  & \cr  \bl &\bl & & & \cr \bl &\bl& \bl & \bl &\bl & & & & }}} \qquad\iff
\quad \; \core(\lambda)= {\tiny{\tableau*[scY]{ \cr \cr  & \cr  &  & \cr  & & & & \cr  && & & & & & & }}}
\end{equation*}
\end{example}

\subsection{Affine symmetric group}
The affine symmetric group $\tilde S_{k+1}$ is generated by the $k+1$ elements $\sigma_0,\dots, \sigma_{k}$
satisfying the affine Coxeter relations:
\begin{eqnarray} \label{coxeter}
\sigma_i^2  =  id, \qquad \sigma_i \sigma_j = \sigma_j \sigma_i \quad \; (i-j\neq \pm 1\mod k+1), \quad\text{and}\quad
\sigma_i \sigma_{i+1} \sigma_{i} =  \sigma_{i+1} \sigma_i \sigma_{i+1} \, .
\end{eqnarray}
Here, and in what follows, $\sigma_i$ is understood as $\sigma_{i \! \! \! \mod  \! k+1}$ if $i\geq k+1$. Elements of
$\tilde S_{k+1}$ are called affine permutations, or simply permutations.
The length of an element $\sigma \in \tilde S_{k+1}$ is the smallest number
$\ell$ such that $\sigma$ can be written as
$\sigma = \sigma_{i_1} \cdots \sigma_{i_{\ell}}$ for some $i_1,\dots,i_{\ell}
\in \{0,1,\dots,k \}$.

Let $S_{k+1}$ be the subgroup of $\tilde S_{k+1}$ generated by $\sigma_1,\dots,\sigma_k$
(and thus isomorphic to the usual symmetric group on $k$ elements).
It is known that the minimal length (left) coset representatives
in the quotient $\tilde S_{k+1}/S_{k+1}$
are in bijective
correspondence with $k+1$-cores (and thus with $k$-bounded partitions).  There is a natural action of the affine
symmetric group on cores that accounts for this relation \cite{[MM],[L]}:
\begin{definition}
\label{kcore} The {\it ``operator $\sigma_i$"} acts on a $k+1$-core by:
\begin{itemize}
\item[{(a)}] removing all removable corners with $k+1$-residue $i$ if
there is at least one removable corner of $k+1$-residue $i$

\item[{(b)}] adding all addable corners with $k+1$-residue $i$ if
there is at least one addable corner with $k+1$-residue $i$

\end{itemize}
\end{definition}
In that correspondence, if $\sigma = \sigma_{i_1} \cdots \sigma_{i_\ell} $ is
a minimal length coset
representative in $\tilde S_{k+1}/S_{k+1}$, then $\sigma_{i_1} \cdots \sigma_{i_\ell}(\emptyset)$ is the corresponding core (where
$\emptyset$ is the empty core).  Given this correspondence, in what follows we will not distinguish between cores and
minimal length coset representatives
in $\tilde S_{k+1}/S_{k+1}$.  We will also write $\sigma \in \tilde
S_{k+1}/S_{k+1}$ to mean that  $\sigma \in \tilde S_{k+1}$ is a minimal
length
coset representative.  Note that $\sigma_i$  is not
defined when there
are no addable or removable corners of residue $i$.
In this case, the corresponding permutation
is not a minimal length coset representative in
$\tilde S_{k+1}/S_{k+1}$, and will thus be of no concern to
us.

\begin{example}
Given $k=3$, the product $\sigma_2 \sigma_3 \sigma_1 \sigma_0 \in \tilde S_{4}/S_{4}$ corresponds to the $4$-core
$\gamma=(3,1,1)$, since:
\begin{equation*}
\sigma_2 \sigma_3 \sigma_1 \sigma_0 (\emptyset)=\sigma_2 \sigma_3 \sigma_1
\left(\, {\tiny{\tableau[scY]{0}}} \, \right)=\sigma_2 \sigma_3 \left(\,
  {\tiny{\tableau[scY]{0,1}}}\, \right) =\sigma_2 \left(\,
  {\tiny{\tableau[scY]{3|0,1}}} \,
\right)={\tiny{\tableau[scY]{2|3|0,1,2}}}
\end{equation*}
\end{example}

Given $r<s \in \mathbb Z$, we can define the transposition $t_{r,s} \in \tilde
S_{k+1}$ by
\begin{equation}\label{trs:first}
t_{r,s} = \sigma_{r} \sigma_{r+1} \cdots \sigma_{s-2} \sigma_{s-1} \sigma_{s-2} \cdots \sigma_{r+1} \sigma_{r}
\end{equation}
It is easy to see that  $t_{r,s}$ is an involution.  Furthermore, if $s-r<k+1$, then it is easy to see using the
Coxeter relations that
\begin{equation} \label{trs:second}
t_{r,s} = \sigma_{s-1} \sigma_{s-2} \cdots \sigma_{r+1} \sigma_{r} \sigma_{r+1} \cdots \sigma_{s-2} \sigma_{s-1} \, .
\end{equation}
Now, given two $k+1$-cores $\delta$ and $\gamma$, we say that $\delta\lessdot
\gamma$ if there exists  a transposition $t_{r,s}$  such that $t_{r,s} \delta=\gamma$ and
$|\kbnd(\gamma)|=|\kbnd(\delta)|+1$.  The transitive closure of this relation corresponds in $\tilde S_{k+1}/S_{k+1}$
to the (strong) Bruhat order.

\begin{example}
Let $k=3$ and $\la=(2,1,1)$. Apply $t_{1,3}$ to $\core(\la)=(3,1,1)=\gamma$.
\begin{equation*}
t_{1,3}(\gamma)=\sigma_1 \sigma_2 \sigma_1 \left( \,{\tiny
    \tableau[scY]{2|3|0,1,2}} \, \right)=\sigma_1 \sigma_2 \left( \,
{\tiny \tableau[scY]{1|2|3|0,1,2}} \, \right)= \sigma_1 \left(\, {\tiny
  \tableau[scY]{1|2|3|0,1}} \, \right) =  {\tiny
\tableau[scY]{2|3|0}} = \delta
\end{equation*}
Also note that $t_{1,3}(\delta)=\gamma$. Since $\kbnd(\delta)=(1,1,1)$, then we have that $(1,1,1)\lessdot (3,1,1)$.
\end{example}

Equivalently, it can be shown that $\delta\lessdot \gamma$ iff $\delta \subset \gamma$ and
$|\kbnd(\gamma)|=|\kbnd(\delta)|+1$.   It is also known
\cite{[MM],[L]} that the
Bruhat order, the transitive closure of this relation, is given by $\delta < \gamma$
iff $\delta \subset \gamma$.

The following result is proved in \cite{[LLMS]}.  Note that a ribbon is a connected
skew-diagram that does not contain any $(2,2)$ subdiagram.
\begin{proposition} \label{T:scover} Let $\delta\lessdot \gamma$ be
$k+1$-cores with
$t_{r,s} \delta=\gamma$ and $0<r<s$. Then
\begin{enumerate}
\item $s-r<k+1$.
\item Each connected component of $\gamma/\delta$ is a ribbon with $s-r$
cells in diagonals of $k+1$-residues ${r},{r+1},\dotsc,{s-1}$.
\item The components are translates of each other and their heads
lie on ``consecutive" diagonals of $k+1$-residue ${s-1}$.
\end{enumerate}
\end{proposition}

\subsection{$k$-Schur functions}
We now present the characterization of $k$-Schur functions given in
\cite{[LMproofs]}.

\begin{definition}
\label{defktabgen}
Let $\gg$ be a $k+1$-core with $m$ $k$-bounded hooks and  let
$\aa=(\aa_1,\ldots,\aa_r)$ be a composition of $m$.
A ``$k$-tableau" of shape $\gg$ and ``$k$-weight" $\aa$
is a filling of $\gg$ with integers $1,2,\ldots,r$ such that

\smallskip
\noindent
(i) rows are weakly increasing and columns are strictly increasing

\smallskip
\noindent
(ii) the collection of cells filled with letter $i$ are labeled by exactly
$\alpha_i$ distinct $k+1$-residues.
\end{definition}

\begin{example}
\label{exssktab}
The $3$-tableaux of $3$-weight $(1,3,1,2,1,1)$ and shape $(8,5,2,1)$ are:
\begin{equation*}
{\tiny{\tableau*[scY]{5\cr 4&6\cr2&3&4&4&6\cr 1&2&2&2&3&4&4&6 }}} \qquad {\tiny{\tableau*[scY]{6\cr 4&5\cr 2&3&4&4&5\cr
1&2&2&2&3&4&4&5 }}} \qquad {\tiny{\tableau*[scY]{4\cr 3&6\cr 2&4&4&5&6\cr 1&2&2&2&4&4&5&6 }}}
\end{equation*}
\end{example}

\medskip

\begin{remark}
\label{ktabtab}
When $k$ is large, a $k$-tableau $T$ of shape $\gamma$ and $k$-weight
$\mu$ is a semi-standard tableau of weight $\mu$
since no two diagonals of $T$ will have the same residue.
\end{remark}

We denote the
set of all $k$-tableaux of shape $\core(\mu)$ and $k$-weight $\aa$
by $\mathcal T^k_{\aa}(\mu)$, and define  the ``$k$-Kostka numbers" as:
\begin{equation}
K_{\mu\alpha}^{(k)}\;=\;|T^k_\aa(\mu)|
\,.
\end{equation}
As is the case for the Kostka number, they
are such that $K_{\mu\alpha}^{(k)}=K_{\mu\gamma}^{(k)}$ if $\gamma$ is
a permutation of $\alpha$, and
satisfy a triangularity property.
\begin{property}
\label{trikostka}
For any $k$-bounded partitions $\lambda$ and $\mu$,
\begin{equation}
K_{\mu\lambda}^{(k)}=0 \quad\text{when}\quad \mu \ntrianglerighteq
\lambda \quad
\text{ and }\quad K_{\mu\mu}^{(k)}=1\,.
\end{equation}
\end{property}

Thus the matrix $||K^{(k)}||_{\mu,\lambda}$ (with $\mu$ and $\lambda$ running
over all $k$-bounded partitions of a given degree) is invertible,
naturally giving rise to a family of symmetric functions.
\begin{definition}
\label{kschurdef}
The ``$k$-Schur functions", indexed by $k$-bounded partitions,
are defined as forming the unique basis of
$\Lambda^{k}=\mathbb Z[h_1,\ldots,h_k]$ such that:
\begin{equation}
\label{e1}
h_\lambda = s_\lambda^{(k)}+\sum_{\mu : \mu\rhd\lambda} K_{\mu\lambda}^{(k)}
s_\mu^{(k)}\,\quad\text{for all } \lambda \text{ such that }
\lambda_1 \leq k\,.
\end{equation}
\end{definition}

\medskip

From this $k$-tableau characterization,
many properties of $k$-Schur
functions are derived in \cite{[LMproofs]}.  Of these properties,
the only one relevant to this work is the $k$-Pieri rule, which we now present
in the form given in \cite{[LLMS]}.

Any {\it proper} subset $A \subsetneq \mathbb Z_{k+1}=\{0,\dots,k \}$ decomposes into unions $I_1\cup I_2\cup \cdots
\cup I_m$ of maximal cyclic intervals. For each cyclic component $[a,b]$ of
$A$, we will let $\sigma_A$ be equal to the product of factors
$\sigma_b \sigma_{b-1} \cdots \sigma_a$ (observe the descending order of the
indices).
For instance, if $k=8$ and $A=
\{0,1,3,4,6,8\}$, we have that the cyclic components are $[8,1]=\{8,0,1\}$, $[3,4]=\{3,4\}$ and $[6,6]=\{ 6\}$.  The
corresponding element $\sigma_A$ is thus equal to $\sigma_{1}\sigma_{0}\sigma_8 \sigma_{4} \sigma_3 \sigma_6= \sigma_4
\sigma_3 \sigma_6 \sigma_{1}\sigma_{0}\sigma_8=\cdots$ (the components commute among themselves).

\begin{proposition} Let $h_\ell$ be the $\ell^{th}$
  complete symmetric function.  Furthermore, let $\lambda$ be a $k$-bounded partition, and
$\gamma=\core(\lambda)$ be its corresponding $k+1$-core. Then, if $\ell \leq
k$, the $k$-Schur functions satisfy the
$k$-Pieri rule:
\begin{equation} \label{kpieriform}
h_\ell \, s_{\lambda}^{(k)} = \sum_{A} s_{\kbnd (\sigma_A(\gamma))}^{(k)}
\end{equation}
where the sum is over all subsets $A$ of $\mathbb Z_{k+1}$ of cardinality $\ell$ such that $\kbnd (\sigma_A(\gamma))$ is a
partition of size $\ell+|\lambda|$.
\end{proposition}
\begin{remark} \label{remarkstrip}
This formulation of the $k$-Pieri rule
is equivalent to the one presented in \cite{[LMproofs]}.
In that case,  the $k$-Pieri rule can be interpreted as
\begin{equation}
h_{\ell} \, s_{\lambda}^{(k)} = \sum_{\mu} s_{\mu}^{(k)} \, ,
\end{equation}
where the sum is over all $k$-bounded partitions $\mu$ such that
$\core(\mu)/\core(\lambda)$ is a horizontal strip with exactly $\ell=|\mu|-|\lambda|$
distinct residues.  In \cite{[LLMS]}, it is shown that this condition is equivalent to
$\core(\mu)$ being equal
to $\sigma_A (\core(\lambda))$ for $A$ a subset of $\mathbb Z_{k+1}$ of cardinality $\ell$.
\end{remark}

\begin{example}
We illustrate the $k$-Pieri rule for $k=6$ by doing the product of
$h_4$ and $s_{(4,3,2,2,2,1)}^{(6)}$. First,
we give
the Ferrers diagram of the $k+1$-core $\core(\lambda)$
associated to $\la=(4,3,2,2,2,1)$, with
residues on the addable positions and
frames on the $k$-bounded cells.
\begin{equation*}
\tiny \tableau[scY]{~1|*,~3|*,*|*,*|*,*,~0,~1|,*,*,*,~3,~4|,,*,*,*,*,~6,~0,~1,~2}
\end{equation*}
We then show the possible subsets $A$
of $\mathbb Z_{k+1}$ of cardinality $\ell = 4$
such that that $\kbnd(\sigma_A(\core(\lambda)))$ is a
partition of 18:
\begin{equation*}
\begin{array}{ccccccc}
{\tiny \tableau[scY]{*1|*,~3|*,*|*,*|,*,*0,*1|,*,*,*,~3,~4|,,,,*,*,*6,*0,*1,*2}} & & &
{\tiny\tableau[scY]{*1|*,*3|*,*|*,*|,*,*0,*1|,,*,*,*3,~4|,,,,*,*,*6,*0,*1,~2}} & & &
{\tiny\tableau[scY]{~1|*,*3|*,*|*,*|*,*,*0,~1|,,*,*,*3,*4|,,*,*,*,*,*6,*0,~1,~2}}
\\
\text{\scriptsize $\{ 6,0,1,2 \}$} & & & \text{\scriptsize $ \{ 6,0,1,3 \}$} & & & \text{\scriptsize $\{ 6,0,3,4 \}$}
\end{array}
\end{equation*}

\begin{equation*}
\begin{array}{ccc c}
{\tiny \tableau[scY]{*1|*,*3|*,*|*,*|*,*,~0,~1|,,*,*,*3,*4|,,*,*,*,*,*6,~0,~1,~2}} & & &
{\tiny\tableau[scY]{*1|*,*3|*,*|*,*|,*,*0,*1|,,*,*,*3,*4|,,*,*,*,*,~6,~0,~1,~2}}
\\
\text{\scriptsize $\{ 6,1,3,4 \}$} & & & \text{\scriptsize $\{ 0,1,3,4 \}$}
\end{array}
\end{equation*}
Therefore,
\begin{equation}
h_4 s_{(4,3,2,2,2,1)}^{(6)}= s^{(6)}_{(6,3,3,2,2,1,1)} + s^{(6)}_{(5,3,3,2,2,2,1)}  + s^{(6)}_{(5,4,3,2,2,2)} +
s^{(6)}_{(5,4,2,2,2,2,1)} + s^{(6)}_{(4,4,3,2,2,2,1)}
\end{equation}

\end{example}

\section{Bernstein operators} \label{sectbern}
The notation used in this section is taken from \cite{[M]}.  For $m$ a nonnegative integer, the operator $e_{m}^{\perp}$ is defined such that given any symmetric functions $f$ and
$g$,
\begin{equation}
\langle e_m^{\perp} \, f, g \rangle = \langle f, e_m \, g \rangle \, ,
\end{equation}
with $e_m$ the $m^{th}$ elementary symmetric function and $\langle \cdot, \cdot \rangle$ the unique scalar product with
respect to which the Schur functions are orthonormal.  It can be shown that the operator $e_m^{\perp}$ has the
following simple action on a Schur function
\begin{equation}
e_m^{\perp} \, s_{\lambda} = \sum_{\mu} s_{\mu} \,
\end{equation}
where the sum is over all partition $\mu$ such that $\lambda/\mu$ is a $m$-vertical strip.

For $n$ a nonnegative integer, the Bernstein operator is \cite{[Z]}
\begin{equation}
B_n = \sum_{i \geq 0} (-1)^i h_{n+i}  \, e_{i}^{\perp} \, ,
\end{equation}
where $h_m$ is the $m^{th}$ complete symmetric function.  The Bernstein operators allow to build the Schur functions
recursively.  That is, for $\lambda= (\lambda_1, \dots, \lambda_\ell)$,
\begin{equation}
B_{\lambda_1} B_{\lambda_2} \cdots B_{\lambda_\ell} \cdot 1 = s_{\lambda} \, .
\end{equation}
Or equivalently, if $\hat \lambda= (\lambda_2, \dots, \lambda_\ell)$, then
\begin{equation} \label{recSchur}
B_{\lambda_1} \, s_{\hat \lambda} = s_{\lambda} \, .
\end{equation}

\section{The main formula}

Note that for the remainder of the article, as it was the case in the previous
section, $\hat \lambda$ will stand for the
partition $\lambda$ without its first part.

Before being able to describe  analogs of these operators for the $k$-Schur
functions, we need some definitions.
Let $\gamma$ be a $k+1$-core, and
let
$x$ be the cell corresponding to the leftmost $k$-bounded cell in the first row of $\gamma$.   If $x$ lies in column
$j$, then let the {\it main subpartition} of $\hat \gamma$ (relatively to $\gamma$), be the subpartition of $\hat
\gamma$  made out of the columns of $\hat \gamma$ from
column $j$ up to column $\gamma_2$ (that is, from column $j$ rightward).
For instance, let $k=6$,  and consider the 7-core
$\gamma=(5,5,3,3,2,2,1,1,1)$.  As illustrated in the following diagram,
where the $k$-bounded cells of
$\gamma$ are in bold face, the leftmost $k$-bounded cell, $x$, in the first row of $\gamma$ is in column 3.  Therefore,
the main subpartition of $\hat \gamma$ is the partition filled with $\circ$'s in the diagram.
$$
{\tiny{\tableau*[scY]{ \tf \cr \tf \cr \tf   \cr \tf & \tf \cr \tf & \tf \cr &
      \tf &  \tf  \cr & \tf &  \tf  \cr
&&\tf  & \tf  & \tf  \cr &&\tf x &\tf &\tf \cr}}} \qquad \implies \qquad {\tiny{\tableau*[scY]{ \tf \cr \tf \cr \tf \cr
\tf & \tf \cr \tf & \tf \cr &
      \tf &  \tf \circ \cr & \tf &  \tf \circ \cr
&&\tf \circ & \tf \circ & \tf \circ \cr}}}
$$
\begin{remark}  It is important to realize that the concept of
main subpartition is only defined for a
$\delta$ such that $\delta=\hat \omega$ for a \underline{given} $k+1$-core
$\omega$.  When using  the term main subpartition
of $\hat \gamma$, it is understood that the larger partition is in this case
$\gamma$.
\end{remark}

\begin{remark} \label{rembounded} The cells in the main subpartition of $\hat \gamma$
are all $k$-bounded.   When $k$ is large enough, the main subpartition of $\hat \gamma$ coincides with $\hat \gamma$.
\end{remark}
\begin{remark} \label{remleft}If in $\hat \gamma$ there are columns to the left of its main
  subpartition, then they are all strictly larger than the largest
column of the main subpartition.  This is because the cell to the left of $x$ in $\gamma$ (see the example above) would
not have otherwise a hook-length larger than $k+1$.
\end{remark}

Recall that $\delta \lessdot \omega$ iff $\delta \subseteq \omega$
and the number of $k$-bounded hooks in $\delta$ is one less than that
in $\omega$.  Also recall from Proposition~\ref{T:scover} that
if $\delta \lessdot \omega$, then $\omega/\delta$ is a union of identical ribbons (of size smaller or equal to $k$)
whose heads (southeast-most cell of the ribbon) occur on consecutive diagonals of a certain $k+1$-residue.
 A ribbon will be horizontal if, as its
name suggests, it coincides with a horizontal partition $(n)$ for some $n$.

\begin{definition}\label{defklstrip}
Let $\gamma$ be a core such that the main subpartition of $\hat \gamma$ is of length $m$. We will say that the core
$\delta$ can be obtained by removing a vertical $(k,\ell)$-strip from $\hat \gamma$ if there exists a sequence of cores
$\hat \gamma=\omega^{(1)}\supset  \omega^{(2)} \supset \cdots \supset \omega^{(\ell+1)}=\delta$ such that
\begin{enumerate}
\item $\omega^{(i+1)} \lessdot \, \omega^{(i)}$  for all $i=1,\dots,\ell$.
\item $\omega^{(i)}/ \omega^{(i+1)}$ is a union of horizontal ribbons, the
lowest of which appears in a row $r_i$ with $1\leq r_i \leq m$.
\item $r_1, \dots, r_\ell$ are all distinct.
\end{enumerate}
\end{definition}

\begin{example} \label{ejemplo vert strip}
Let $k=5$ and consider $\gamma=(6,6,3,3,3,1,1,1,1)$. It can be checked
that the length of the
main subpartition of $\hat \gamma$ is $m=2$.  The $k+1$-core
$\delta=(5,4,3,2,1,1,1,1,1)$ can be obtained from $\hat \gamma$, by
removing the following
vertical $(5,2)$-strip :
\begin{equation*}
\hat{\gamma}=\omega^{(1)}= (6,6,3,3,3,1,1,1,1) \supset \omega^{(2)} = (6,4,3,3,1,1,1,1,1) \supset \omega^{(3)} =
(5,4,3,2,1,1,1,1,1)=\delta.
\end{equation*}
In the following sequence of Ferrers diagrams, we see that all the conditions for a a vertical $(5,2)$-strip are
satisfied. The framed cells correspond to successive ribbons having their
lowest occurrence in different rows and within
the first $m=2$ rows.
\begin{equation*}
{\tiny{\tableau[scY]{||||,*,*|,,|,,|,,,,*,*|,,,,,}}} \quad \supset \quad
{\tiny{\tableau[scY]{||||,~,~|,,*|,,|,,,,~,~|,,,,,*}}} \quad \supset \quad
{\tiny{\tableau[scY]{||||,~,~|,,~|,,|,,,,~,~|,,,,,~}}}
\end{equation*}
\end{example}

\begin{remark} \label{remlowest} By Remark~\ref{remleft}, in condition (2)
  of Definition~\ref{defklstrip},
the lowest ribbons
 are always contained
entirely in the main subpartition of $\hat \gamma$.
\end{remark}
\begin{lemma}\label{lemmaorder}  Suppose we have a vertical $(k,\ell)$-strip $\hat \gamma=\omega^{(1)}\supset  \omega^{(2)}
\supset \cdots \supset \omega^{(\ell+1)}=\delta$, whose lowest ribbons occur in rows $r_1,\dots,r_\ell$.  Then, there
exists a sequence $\hat \gamma=\bar \omega^{(1)}\supset  \bar \omega^{(2)} \supset \cdots \supset \bar
\omega^{(\ell+1)}=\delta$, whose lowest ribbons occur in rows $\bar r_1 > \cdots > \bar r_\ell$.  That is, removing a
vertical $(k,\ell)$ strip can always be done in a certain order (by removing the ribbon whose lowest ribbon is the
highest, then the one whose lowest ribbon is the second highest, and so on).
\end{lemma}
\begin{proof}
Suppose we have $\omega^{(i+1)} \lessdot \omega^{(i)} \lessdot \omega^{(i-1)}$, with both $\omega^{(i-1)}/\omega^{(i)}$
and $\omega^{(i)}/\omega^{(i+1)}$ given by a union of horizontal ribbons, the lowest of which are respectively $R_1$
and $R_2$. By Proposition~\ref{T:scover}, we have $\omega^{(i)}= t_{r,s} (\omega^{(i-1)})$, where $r$ (resp. $s-1$) is
the residue of the leftmost (resp. rightmost) cell in $R_1$.  Similarly, we have $\omega^{(i+1)}= t_{r',s'}
(\omega^{(i)})$, where $r'$ (resp. $s'-1$) is the residue of the leftmost (resp. rightmost) cell in $R_2$. If $R_1$
does not sit on top of $R_2$ (in which case $R_1$ would necessarily have to be removed first), we have that the cyclic
intervals $[r,s-1]$ and $[r',s'-1]$ are disjoint and not contiguous. This is because $R_1$ and $R_2$ belong to the main
subpartition of $\hat \gamma$ (which does not have repeated diagonals of the same residue) and because   $r\neq s' \mod
k+1$ (otherwise there would be a hook of length $k+1$ in the core  $\omega^{(i-1)}$). Observe that in this case,
$$
t_{r',s'} t_{r,s}(\omega^{(i-1)}) = t_{r,s} \left( t_{r',s'}(\omega^{(i-1)}) \right) \lessdot
t_{r,s}(\omega^{(i-1)}) \implies  t_{r',s'} (\omega^{(i-1)}) \lessdot
\omega^{(i-1)}
$$
and
$$
t_{r,s} (\omega^{(i-1)}) \lessdot
\omega^{(i-1)} \implies
t_{r',s'} t_{r,s}(\omega^{(i-1)}) = t_{r',s'} \left( t_{r,s}(\omega^{(i-1)}) \right) \lessdot
t_{r',s'}(\omega^{(i-1)})   \, ,
$$
since there is no interference in the adding and deleting process involved
in acting with the transpositions.
Therefore,
$$\omega^{(i+1)}=t_{r',s'} t_{r,s}(\omega^{(i-1)}) \lessdot
t_{r,s}(\omega^{(i-1)}) \lessdot \omega^{(i-1)}$$ leads to
$$\omega^{(i+1)}=t_{r',s'} t_{r,s}(\omega^{(i-1)}) \lessdot
t_{r',s'}(\omega^{(i-1)}) \lessdot \omega^{(i-1)} \, ,$$ meaning that the ribbons $R_1$ and $R_2$ (with their
translates) can be removed in any order.   The general result then follows by
applying
this idea again and again.
\end{proof}

The following results will show that removing a $(k,\ell)$-strip from
a $k+1$-core $\hat \gamma$ removes a $\ell$-vertical
strip from the $k$-bounded partition associated to $\hat \gamma$.

\begin{lemma}  Let $\gamma$ and $\delta$ be two $k+1$-cores such that $|\kbnd (\gamma)|=
 |\kbnd (\delta)|+1$ and such that $\delta \subset \gamma$.  If
 $\gamma/\delta$ is a union of horizontal ribbons, the highest of which appears in
 row $i$, then $\kbnd(\gamma)=\kbnd(\delta)+{\bf e}_i$, where ${\bf e}_i$ is
 the vector with a 1 in position $i$ and 0 everywhere else.
\end{lemma}
\begin{proof}  First note that if a row of $\gamma/\delta$ does not contain a
  horizontal ribbon then the number of $k$-bounded cells in that row is the same in $\gamma$ and
in $\delta$.  This is because the hook-length of a cell in that row is changed by at most one cell, preventing a change
of the hook length from more than $k+1$ to less than $k+1$ (recall that a $k+1$-core does not contain cells with
hook-lengths of $k+1$). As for the remaining rows, recall that the head of the ribbons occur in consecutive diagonals
of the same residues. The proof is then illustrated in the
following example at $k=3$, where the cells in bold face
are the horizontal ribbons in $\gamma/\delta$, and the cells with an $x$ are
the cells that went from not being $k$-bounded in $\gamma$ to being $k$-bounded
in $\delta$.
\begin{equation*}
{\tiny{\tableau*[scY]{ \cr & \cr & &  \cr & &  &   \cr & & & x & & \tf & \tf \cr & & & & &
      x & x & & \tf & \tf \cr & & & & & & & & x & x & & \tf &  \tf \cr}}}
\end{equation*}
One simply needs to observe that in each row that contains a ribbon,
the number of cells with an $x$ is equal to the number of cells in bold face
(except in the highest such row).
Since $|\kbnd (\gamma)|=
 |\kbnd (\delta)|+1$, this implies that it must differ by one in the highest
row that contains a ribbon.
\end{proof}

\begin{proposition} \label{vertstripprop}
 Let $\delta$ and $\hat \gamma$ be $k+1$-cores
such that $\delta$
can be obtained by removing a vertical $(k,\ell)$-strip from $\hat \gamma$.
Then $\kbnd(\hat \gamma)/\kbnd(\delta)$ is a
vertical $\ell$-strip (in the usual sense).
\end{proposition}
\begin{proof}
From the previous lemma, we simply need to show that when going
  from $\hat \gamma$ to $\delta$ by removing horizontal
  ribbons, two ribbons will never occur in the same row.
From Lemma \ref{lemmaorder}, it is possible to choose $\hat
\gamma=\omega^{(1)}\supset  \omega^{(2)} \supset \cdots \supset \omega^{(\ell+1)}=\delta$ such that the horizontal ribbons
are removed from top to bottom in the main subpartition. Let
$\omega^{(i)}/\omega^{(i+1)}$ contain a given ribbon $\tilde R$ and all
its translates. The rightmost cell of $R$, the translate of $\tilde R$ in
the main subpartition, has a residue $r$
that is not contained in any ribbon above it in the main subpartition (from Remark \ref{remlowest} and
\ref{rembounded}). Therefore, when going from $\hat \gamma$ to $\omega^{(i)}$,
no cells of residue $r$ are removed or added, and
thus the rightmost cell of $\tilde R$ is also extremal in $\hat \gamma$.
By Proposition~\ref{propo4.1},
this means that the rightmost cell of $\tilde R$
has to be at the
end of its row in $\hat \gamma$ since the rightmost cell of $R$ is also at
the end it its row in $\hat \gamma$ (no two ribbons can occur in the same
row of the main subpartition by definition of $(k,\ell)$-strip).
Therefore, no horizontal ribbons can ever occur
to the right of $\tilde R$.
\end{proof}

We can now define the recursion for $k$-Schur functions that extends formula
\eqref{recSchur}.
\begin{definition}
Let $\mathcal
V^{(k,r)}$ be the $\mathbb Z$-linear
span of $k$-Schur functions whose first part is not larger than $r$.   Given a partition $\nu$
such that $s_{\nu}^{(k)} \in \mathcal V^{(k,r)}$, let $\lambda=(r,\nu_1,\nu_2,\dots)$ and $\gamma =\core(\lambda)$.
Then  the linear operator $e_{\ell,r}^{\perp}$ is defined on $\mathcal V^{(k,r)}$
to be such that
\begin{equation}
e_{\ell,r}^{\perp} \, s_{\nu}^{(k)} = \sum_{\mu}  s_{\mu}^{(k)} \, ,
\end{equation}
where the sum is over all $k$-bounded partition $\mu$ such that $\core(\mu)$ can be obtained by removing a vertical
$(k,\ell)$-strip from $\hat \gamma$.  If there is no such $\mu$, the result is
simply defined to be zero.
\end{definition}
Note that we only use the symbol $e_{\ell,r}^{\perp}$ in
analogy with $e_{\ell}^{\perp}$.  That is, to the best
of our knowledge,  $e_{\ell,r}^{\perp}$
is not the adjoint of multiplying by some
symmetric function
$e_{\ell,r}$ with
respect to any
scalar product.
We
should also point out that the operator $e_{\ell,r}^{\perp}$ does in fact depends on $r$, since $r$ appears in the
definition of the $k+1$-core $\gamma$ (and since extracting a $(k,\ell)$-strip from $\hat \gamma$ actually depends
on $\gamma$).
\begin{remark}  By Proposition~\ref{vertstripprop}, the $\mu$'s
such that $s_{\mu}^{(k)}$ occur in
in the action of $e_{\ell,r}^{\perp}$ on $s_{\nu}^{(k)}$ are
such that $\nu/\mu$ is a vertical $\ell$-strip.   These $\mu$'s
are thus a
subset of  the $\mu$'s such
that $s_{\mu}$ occur
in the action of $e_{\ell}^{\perp}$ on $s_{\nu}$.
\end{remark}

\begin{example} \label{ejemplo sacar todas strips}
Let $\nu = (4,3,2,2,1)$ and $k=6$. Then $\la = (4,4,3,2,2,1)$ and
$\gamma=\core(\la)=(6,6,3,2,2,1)$.  Hence
\begin{equation*} \label{diagrama tableau ejemplo}
{\hat \gamma}=(6,3,2,2,1)={\tiny \tableau[scY]{|,|,|,,*|,,*,*,*,*}}
\end{equation*}
where the framed cells correspond to the main subpartition. If we apply $e_{1,4}^{\perp}$, the vertical
$(k,\ell)$-strips need to be of length $\ell=1$.
The following diagrams show the vertical $(k,1)$-strips that can be
obtained, with the $k$-bounded cells marked with $\circ$
(thus corresponding to the $k$-bounded partitions).
\begin{equation*}
{\tiny{\tableau[scY]{|,*|,|,,|,,,,,*}}} \quad \supset \quad
{\tiny{\tableau[scY]{\circ|\circ|\circ,\circ|\circ,\circ,\circ|,\circ,\circ,\circ,\circ}}}
\end{equation*}

\begin{equation*}
{\tiny{\tableau[scY]{|,|,|,,*|,,,,,}}} \quad \supset \quad
{\tiny{\tableau[scY]{\circ|\circ,\circ|\circ,\circ|\circ,\circ|,,\circ,\circ,\circ,\circ}}}
\end{equation*}

\begin{equation*}
{\tiny{\tableau[scY]{|,|,|,,|,,,,*,*}}} \quad \supset \quad
{\tiny{\tableau[scY]{\circ|\circ,\circ|\circ,\circ|\circ,\circ,\circ|,\circ,\circ,\circ}}}
\end{equation*}
From here, we obtain that
\begin{equation}
e_{1,4}^{\perp} \, s_{\nu}^{(6)} = s_{(4,3,2,1,1)}^{(6)} +
s_{(4,2,2,2,1)}^{(6)} + s_{(3,3,2,2,1)}^{(6)}\, .
\end{equation}
\end{example}

Now, for $r=1,\dots,k$, let
\begin{equation}
B_r^{(k)} = \sum_{\ell \geq 0} (-1)^\ell h_{r+\ell} \, e_{\ell,r}^{\perp} \, .
\end{equation}
Note that this operator is only defined on $\mathcal V^{(k,r)}$.
The main result of this article is then the following.
\begin{theorem} \label{kBernsteinThm} Let
  $\lambda=(\lambda_1,\lambda_2,\dots)$
be a $k$-bounded partition.  Then
\begin{equation} \label{kBernsteinEq}
B_{\lambda_1}^{(k)}  \,  s_{\hat \lambda}^{(k)} = s_{\lambda}^{(k)} \, .
\end{equation}
\end{theorem}

\begin{example} \label{ejemplo expancion en bernstein}
Using $k$ and $\la$ as in Example~\ref{ejemplo sacar todas strips}, we will show that:
\begin{equation}
B_4^{(6)}s_{\hat \la}^{(6)} = s_{\la}^{(6)}.
\end{equation}
By definition, our equation amounts to:
\begin{equation} \label{equationexpanded}
B_4^{(6)} s_{\hat \la}^{(6)} = \sum_{\ell \geq 0} (-1)^\ell h_{4+\ell} \, e_{\ell,4}^{\perp} s_{\hat \la}^{(6)}.
\end{equation}
According to the diagram of equation \eqref{diagrama tableau ejemplo},
we only need to consider vertical $(k,\ell)$-strips up to
$\ell =2$.  For the first term in the sum,
acting with $e_{0,4}^{\perp}$ on $s_{\hat \la}^{(6)}$ gives $s_{\hat \la}^{(6)}$,
since we have to extract a vertical $(k,0)$-strip (which amounts to doing
nothing).
The action of $e_{1,4}^{\perp}$ was
explained in example \ref{ejemplo sacar todas
strips}. To compute the action of $e_{2,4}^{\perp}$ on $s_{\hat \la}^{(6)}$,
we present here the diagrams of the cores
that can be obtained by removing a vertical $(k,2)$-strips from $\hat \la$:
\begin{equation*}
{\tiny{\tableau[scY]{|,|,|,,*|,,,,,}}} \quad \supset \quad {\tiny{\tableau[scY]{|,*|,|,|,,,,,*}}} \quad \supset \quad
{\tiny{\tableau[scY]{\circ|\circ|\circ,\circ|\circ,\circ|,\circ,\circ,\circ,\circ}}}
\end{equation*}

\begin{equation*}
{\tiny{\tableau[scY]{|,|,|,,*|,,,,,}}} \quad \supset \quad {\tiny{\tableau[scY]{|,|,|,|,,,,*,*}}} \quad \supset \quad
{\tiny{\tableau[scY]{\circ|\circ,\circ|\circ,\circ|\circ,\circ|,\circ,\circ,\circ}}}
\end{equation*}
Therefore, \eqref{equationexpanded} gives:
\begin{equation}
B_4^{(6)} s_{\hat \la}^{(6)} = h_4 (s_{\hat \la}^{(6)}) - h_5 (s_{(4,3,2,1,1)}^{(6)} + s_{(4,2,2,2,1)}^{(6)} +
s_{(3,3,2,2,1)}^{(6)}) + h_6 (s_{(4,2,2,1,1)}^{(6)} + s_{(3,2,2,2,1)}^{(6)})
\end{equation}
Now, by applying the $k$-Pieri rule we get:
\begin{equation*}
\begin{array}{rl}
B_4^{(6)} s_{\hat \la}^{(6)} &= s_{4,4,3,2,2,1}^{(6)} + s_{5,4,3,2,1,1}^{(6)} + s_{6,3,3,2,1,1}^{(6)} + s_{5,3,3,2,2,1}^{(6)}
 + s_{6,3,3,2,2}^{(6)} + s_{5,4,2,2,2,1}^{(6)} \\ &- s_{6,3,3,2,1,1}^{(6)} - s_{5,4,3,2,1,1}^{(6)}
 - s_{6,3,2,2,2,1}^{(6)} - s_{5,4,2,2,2,1}^{(6)} - s_{6,4,2,2,1,1}^{(6)}
 - s_{6,3,3,2,2}^{(6)} - s_{5,3,3,2,2,1}^{(6)} \\
 &+ s_{6,4,2,2,1,1}^{(6)} + s_{6,3,2,2,2,1}^{(6)}.
\end{array}
\end{equation*}
Finally, canceling the expression gives $B_4^{(6)} s_{\hat \la}^{(6)}
= s_{4,4,3,2,2,1}^{(6)}=s_{\la}^{(6)} $.
\end{example}

The proof of the theorem will be of a combinatorial nature, and
will ultimately rely on the construction of a sign-reversing involution.
But first, we introduce some notation.
\begin{definition}  Given $\gamma= \core(\lambda)$, let ${\mathcal D}_{\lambda}^{(k)}$
be the set of pairs $(\delta,A)$ such that for some $\ell=0,\dots,k-\lambda_1$:
\begin{enumerate}
\item $\hat \gamma/\delta$ is a removable vertical $(k,\ell)$-strip
\item $A$ is a subset of $\mathbb Z_{k+1}$ of cardinality $\lambda_1+\ell$
such that $\sigma_{A}(\delta)$ satisfies
$|\kbnd(\sigma_A(\delta))|=|\lambda|$ (that is, $\sigma_{A}(\delta)$ is a
$k+1$-core whose number of $k$-bounded cells is $|\lambda|$).
\end{enumerate}
\end{definition}
A pair $(\delta,A) \in \mathcal D_{\lambda}^{(k)}$ can be thought of as the Ferrer's diagram of $\hat \gamma$ with the
cells of $\hat \gamma/\delta$ marked with an $O$ combined with the cells of $\sigma_A(\delta)/\delta$ marked with an
$X$. We will refer to such a diagram as the $OX$ diagram associated to the pair $(\delta,A)$. Observe by
Remark~\ref{remarkstrip} that an $OX$ diagram can never have two $X$'s in the same column since
$\sigma_A(\delta)/\delta$ is a horizontal strip. Cells that contain a $O$ and an $X$ will be called $OX$ cells and
represented by ${ \tableau[scY]{ \bl \overlay X \overlay O \overlay }}$ in diagrams.

\begin{example} \label{ejemplo par de OX diagram}
In example \ref{ejemplo expancion en bernstein}, when we expand the products, we see that the $k$-Schur function
$s_{(6,3,2,2,1)}^{(6)}$ appears two times, first in the product $h_5
s_{(4,2,2,2,1)}^{(6)}$ and
then in the product
$h_6 s_{(3,2,2,2,1)}^{(6)}$.  If we consider the $(\delta,A)$ pairs
and their corresponding $OX$ diagrams associated to
these two occurrences of $s_{(6,3,2,2,1)}^{(6)}$,
we see that the first one is:
\begin{equation*}
((6,2,2,2,1),\{ 4,6,0,1,2 \})\quad  \longleftrightarrow \quad {\tiny \tableau[scY]{~X|,~X|,|,|,, \overlay X \overlay O \overlay ,~X
| , ,,,,,~X,~X,~X,~X }}
\end{equation*}
while the second one is:
\begin{equation*}
 ((4,2,2,2,1),\{ 4,5,6,0,1,2 \}) \quad \longleftrightarrow \quad {\tiny \tableau[scY]{~X|,~X|,|,|,, \overlay X \overlay O \overlay ,~X | ,
,,,\overlay X \overlay O \overlay,\overlay X \overlay O \overlay,~X,~X,~X,~X }}
\end{equation*}
\end{example}

\begin{definition}
Let $(\delta,A) \in \mathcal D_{\lambda}^{(k)}$, and let $\gamma= \core(\lambda)$. A \underline{changeable cell} of the
pair $(\delta,A)$ is a cell of the core $\sigma_A(\delta)$ that $i)$ lies at the top of its column and $ii)$ belongs to
$\hat \gamma$.  The set of pairs $(\delta,A) \in \mathcal D_{\lambda}^{(k)}$ such that $(\delta,A)$ has a changeable
cell will be denoted ${\mathcal C}_{\lambda}^{(k)}$.
\end{definition}

\begin{remark}  In the language of the $OX$ diagram associated to
$(\delta,A)$, a changeable cell is a cell of $\sigma_A(\delta)$ at the top of its
column that is either empty or filled with an $OX$.
\end{remark}

\begin{example}
In example \ref{ejemplo par de OX diagram}, the changeable cells of the two
$OX$ diagrams
are located in
the same positions: in the fifth and sixth cells of the first row and in
the third cell of the second row.
\end{example}

\noindent{\it Proof of Theorem~\ref{kBernsteinThm}.}
Equation \eqref{kBernsteinEq} can be rewritten as
\begin{equation}
\sum_{\ell\geq 0} (-1)^\ell h_{\lambda_1+\ell} \, e_{\ell,\lambda_1}^{\perp} \, s^{(k)}_{\hat \lambda} =
s^{(k)}_{\lambda} \, .
\end{equation}
Using the action of $h_{\lambda_1+\ell}$ and $e_{\ell,\lambda_1}^{\perp}$ on
$k$-Schur functions, this is equivalent to
\begin{equation} \label{eq45}
\sum_{(\delta,A) \in \mathcal D_{\lambda}^{(k)}} (-1)^{|A|-\lambda_1} s_{\kbnd(\sigma_A(\delta))}^{(k)} =
s_{\lambda}^{(k)} \, .
\end{equation}
Let $\gamma= \core(\lambda)$.
We will show in Lemma~\ref{UniquePair} that the pair $(\hat \gamma,B)$,
where $B$ is the subset of $\mathbb Z_{k+1}$ of size $\lambda_1$ such that
$\sigma_B(\hat \gamma)=\gamma$, is the unique pair of
${\mathcal D}_{\lambda}^{(k)}$ that does not have a changeable cell. Given
that the pair $(\hat \gamma,B)$ corresponds to a term $+s_{\lambda}^{(k)}$
in the l.h.s. of \eqref{eq45}, to prove Theorem \ref{kBernsteinThm}
is suffices to show that
\begin{equation}
\sum_{(\delta,A) \in \, \mathcal C_{\lambda}^{(k)}} (-1)^{|A|} s_{\kbnd(\sigma_A(\delta))}^{(k)} = 0 \, ,
\end{equation}
where we recall that $\mathcal C_{\lambda}^{(k)}$ is the set of $(\delta,A) \in \mathcal D_{\lambda}^{(k)}$  that have
 a changeable cell.

 This result will readily follow if there exists an involution $\varphi:
 \mathcal C_{\lambda}^{(k)}\to \mathcal C_{\lambda}^{(k)}$ that
maps the pair $(\delta,A)$ to a pair
$(\delta',A')$ such that $\sigma_A(\delta)=\sigma_{A'}(\delta')$ and
$(-1)^{|A|} = - (-1)^{|A'|}$.
Such a sign-reversing involution will be constructed in the next section
(see Definition~\ref{definitioninvo} and Proposition~\ref{Propofinal}).
\hfill $\square$

By applying Theorem~\ref{kBernsteinThm} again and again, we obtain
a combinatorial interpretation for the expansion coefficients of $k$-Schur
functions in terms of homogeneous symmetric functions.
\begin{corollary}   Let $\lambda$ be a $k$-bounded partition such that $\core(\lambda)=\gamma$.
Suppose that the sequence $S= (\gamma^{(0)},\dots,\gamma^{(r)})$, with $\emptyset= \gamma^{(0)} \subsetneq \gamma^{(1)}
  \subsetneq \cdots \subsetneq \gamma^{(r)}=\gamma$, is such that
for all $i=1,\dots,r$, we have that
 $\gamma^{(i-1)}$ can be obtained by removing a $(k,\ell_i)$-vertical strip from
$\hat \gamma^{(i)}$ for some $\ell_i \in \{0,\dots,k\}$.
Define ${\rm part}(S)$ to be the
partition corresponding to the rearrangement of the sequence
$(\ell_1+p_1,\dots,\ell_r+p_r)$,
where $p_i$ is the length of the first row  of
$\kbnd(\gamma^{(i)})$ (equivalently, $p_i$ is the number of $k$-bounded
cells in the first row of $\gamma^{(i)}$).  Finally, define ${\rm sgn}(S)$ to
be $(-1)^{\ell_1+ \cdots + \ell_r}$.  Then
\begin{equation}
s_{\lambda}^{(k)}  =\sum_{S}  {\rm sgn}(S) \, h_{{\rm part}(S)} \, ,
\end{equation}
where the sum is over all possible sequences $S$ of the form given
above.
\end{corollary}
\begin{example}  The sequences $S$ in the previous corollary can be
interpreted as certain fillings of $\gamma=\core(\lambda)$, as we will illustrate with an example. If $k=4$ and
$\lambda=(2,2,2,1)$, the possible sequences $S$ are seen to be in correspondence with the following fillings of
$\gamma=(3,2,2,1)$:
\begin{equation*}
\begin{array}{ccccccccccccc}
{\tiny \tableau[scY]{1|2,2|3,3|4,4,4}} & & {\tiny \tableau[scY]{*1|1,1|2,2|3,3,3}} & & {\tiny
\tableau[scY]{*2|1,1|2,2|3,3,3}} & & {\tiny \tableau[scY]{1|2,3*|3,3|4,4,4}} & & {\tiny
\tableau[scY]{1*|1,2*|2,2|3,3,3}} & & {\tiny \tableau[scY]{2*|1,2*|2,2|3,3,3}} & & {\tiny \tableau[scY]{1|2,4*|3,3|4,4,4}}
\\
\\{\tiny \tableau[scY]{1*|1,3*|2,2|3,3,3}} & & {\tiny \tableau[scY]{2*|1,3*|2,2|3,3,3}}
& & {\tiny \tableau[scY]{1*|1*,2*|1,1|2,2,2}} & & {\tiny \tableau[scY]{1|2,4*|3,4*|4,4,4}} & & {\tiny
\tableau[scY]{1*|1,*3|2,*3|3,3,3}} & & {\tiny \tableau[scY]{2*|1,3*|2,*3|3,3,3}} & & {\tiny
\tableau[scY]{1*|1*,2*|1,2*|2,2,2}}
\end{array}
\end{equation*}
In these diagrams, the partition $\gamma^{(i)}$ for a sequence $S= (\gamma^{(0)},\dots,\gamma^{(r)})$ can be obtained
by reading the subdiagram containing the letters up to $i$ in the diagram.  The framed cells containing letter $i$
indicate the location of the $(k,\ell_i)$-vertical strip extracted from $\hat \gamma^{(i)}$ to obtain $\gamma^{(i-1)}$.
For instance, the fifth diagram of the second row, corresponds to the sequence $S'=(\emptyset , (1,1), (1,1,1),
(3,2,2,1))$. We illustrate this with the following figure:
\begin{equation*}
\emptyset \quad \stackrel{1}{\longleftarrow} \quad {\tiny \tableau[scY]{1*|\bf 1}} \quad \stackrel{0}{\longleftarrow}
\quad {\tiny \tableau[scY]{1*|1|\bf 2}} \quad \stackrel{2}{\longleftarrow} \quad {\tiny
\tableau[scY]{1*|1,3*|2,3*|3,\bf 3,\bf 3}}
\end{equation*}
Each diagram corresponds to a $k+1$-core in the sequence, and the number above each arrow indicates the size of the
$(k,\ell)$-vertical strip that is extracted to obtain the
$k+1$-core that follows in the sequence. The bold face numbers in each diagram
correspond to the $k$-bounded cells in each first row. We then
see that $\rm{part}(S')=(4,2,1)$, coming from the
composition $(1+1,1+0,2+2)$, and $sgn(S')=(-1)^{1+0+2}=-1$.

We thus have
\begin{eqnarray*}
s_{(2,2,2,1)}^{(4)}& = & h_2h_2h_2h_1 - h_3h_2h_2 - h_3h_2h_2 - h_3h_2h_1h_1 + h_3h_2h_2 +  h_4h_2h_1 - h_3h_2h_1h_1  \\
                   &   & +h_3h_2h_2 + h_3h_3h_1 - h_4h_3 + h_4h_1h_1h_1 - h_4h_2h_1 - h_4h_2h_1 + h_4h_3 \\
                   & = & h_2h_2h_2h_1 - 2 h_3h_2h_1h_1 +h_3 h_3 h_1+ h_4h_1h_1h_1 - h_4h_2h_1
\end{eqnarray*}

\end{example}

\section{The involution}

\begin{lemma} \label{UniquePair}  The only pair in $\mathcal D_{\lambda}^{(k)}$
whose corresponding
  diagram does not have a changeable cell is $(\hat \gamma,B)$, where $B$ is
  the unique subset of $\mathbb Z_{k+1}$ of size $\lambda_1$ such that $\sigma_B(\hat \gamma)=\gamma$.
\end{lemma}

\begin{proof}  It was proven in \cite{[LMproofs]} that the $k$-Schur functions obey
\begin{equation}
h_\ell s_{\nu}^{(k)} = s_{(\ell,\nu)}^{(k)} + \sum_{\mu} s_{\mu}^{(k)}\,
\end{equation}
where the sum is over some $\mu$'s that are larger than $(\ell,\nu)$ in dominance order. That there exists a unique
subset $B$ of $\mathbb Z_{k+1}$ of size
  $\lambda_1$ such that $\sigma_B(\hat \gamma)=\gamma$ follows from that equation
when translated in the language of cores (see the $k$-Pieri rule \eqref{kpieriform}).
By the construction of the core associated to a $k$-bounded partition (see
Example~\ref{exskew}),  $\sigma_B(\hat \gamma) =\gamma$ corresponds to $\hat
\gamma$ with cells added in all the columns of $\hat \gamma$ plus possibly
some extra columns.  Therefore, the $OX$ diagram of $(\hat \gamma,B)$
has an $X$ at the top of every column of $\sigma_B(\hat \gamma)=\gamma$.
Since there are no $O$ cells (no cells were removed), there are no $OX$ cells
and thus there are no changeable cells.

The main observation in the previous paragraph is that acting with $\sigma_B$ adds
a cell on top of every column of $\hat \gamma$. If $|A|=|B|$ and
$A$ adds a cell on top of every column of $\hat \gamma$, then it is easy to
see that we must have $A=B$.  This is because in this case, starting from the
second row, the cores
$\sigma_A(\hat \gamma)$ and $\sigma_B (\hat \gamma)$ are equal since
$\sigma_A(\hat \gamma)/\hat \gamma$ and $\sigma_B(\hat \gamma)/\hat \gamma$ are horizontal
strips.
To have the same number of $k$-bounded cells,
$\sigma_A(\hat \gamma)$ and $\sigma_B(\hat \gamma)$ must thus be equal, which gives that
$A=B$.   Therefore, if $A\neq B$ and $|A|=|B|$,
when acting with $\sigma_A$ on
${\hat \gamma}$, some columns of
 $\hat \gamma$ will not contain an $X$, and thus some changeable
 cells will be present.

If $\delta
\neq \hat \gamma$, then the $OX$ diagram of $(\delta,A)$ will contain
some $O$ cells.   The only possible case without changeable cells is the case
where, weakly to the right of a certain column $c$, the columns are entirely filled with
$O$ cells (since a cell below an $O$ cell is changeable, as is an $OX$ cell),
and where $X$ cells appear on top of every column of $\hat \gamma$
to the left of column $c$.
\begin{center}
\includegraphics{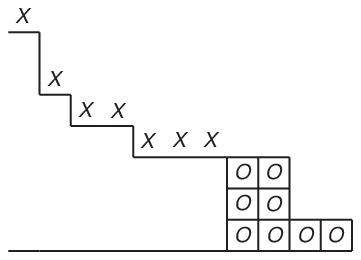}
\end{center}
But this is impossible: from the previous paragraph there are at most $|B|$
distinct residues in the columns above $\hat \gamma$, and we have to add
$|A|>|B|$ residues in a subset of those columns.
\end{proof}

\begin{remark} \label{PositionOfBResidues}
The elements of $B$ are the residues of the $X$'s that sit on top of the main
subpartition of $\hat \gamma$ plus the residues of the $X$'s to the right
of the main subpartition.  This is because by definition of the main
subpartition of $\hat \gamma$, the $\lambda_1$ $k$-bounded cells in the first
row of $\gamma$ start exactly in the leftmost column of the main subpartition.
There are thus exactly $\lambda_1=|B|$ distinct residues at the top of the
columns of $\gamma$ starting from the leftmost column of the main subpartition.
\end{remark}

\begin{example}
Let $k=6$, $\la=(4,4,3,3,2,1,1)$ and $\gamma=\core (\la)$. The following
Ferrers diagrams illustrate Remark
\ref{PositionOfBResidues}. The first is the Ferrers diagram of $\gamma$, where the $k$-bounded cells are given with their
residues. The second is the $OX$ diagram of the unique pair $(\hat \gamma,B)$ that has no changeable cells, with the
main subpartition in framed cells:
\begin{equation*}
\gamma={\tiny \tableau[scY]{1|2|3,4|4,5,6|,6,0,1|,,1,2,3,4|,,,3,4,5,6}} \qquad \qquad (\hat \gamma,B)= {\tiny
\tableau[scY]{~ X||,~ X|,,~ X|,,,~\bf X|,
,,\tf,~\bf X,~\bf X|,,,\tf,\tf,\tf,~\bf X}}
\end{equation*}
In this case, $B=\{1,3,4,6 \}$, recovered from the bold faced $X$ in the second diagram.
\end{example}

\begin{lemma} \label{PositionOfChangeableCell}
Let $(\delta,A) \in \mathcal C_{\lambda}^{(k)}$.  Then, in the $OX$ diagram
associated to $(\delta,A)$, there are no $O$ cells to the right of the
rightmost changeable cell in $(\delta,A)$.  Furthermore, the
rightmost changeable cell in $(\delta,A)$ is in the main
subpartition of $\hat \gamma$.
\end{lemma}
\begin{proof}  Let $c$ be the rightmost changeable cell in $(\delta,A)$.
 By definition of a changeable cell, to the right of $c$
in the $OX$ diagram of $(\delta,A)$,
every column either does not contain any $O$ and finishes with an $X$ or is
entirely made out of $O$'s.  Also observe that, to the right of $c$,
no $X$ cells can appear to the right of an $O$ cell.

Suppose there are some $O$ cells to the right of $c$.
We have in this case $\delta \subset \hat \gamma$.
From the previous comment, the last
column of $\delta$ is entirely made out of $O$'s.
 Now, since $\sigma_A(\delta)/\delta$ is
a horizontal strip and, as we have seen in the proof of Lemma~\ref{UniquePair},
$\gamma/\hat \gamma$ has a cell over all columns of $\hat \gamma$, we have
that
$\sigma_A(\delta) \subset \gamma$ (no $X$ cells can appear
to the right of an $O$ cell, and thus no $X$ cell will appear in the columns
to the right of $\delta \subset \hat \gamma$).  But this implies that $\sigma_A(\delta) < \gamma$
in the Bruhat order, and thus $|\kbnd(\sigma_A(\delta) )|< |\kbnd(\gamma)|$ which
is a contradiction.  This gives the first claim.

In the case where there are some $O$ cells, by the definition of
the pair $(\delta,A)$ there are some $O$ cells in the main subpartition
of $\hat \gamma$, and thus
 the first claim immediately
implies that the rightmost changeable cell in $(\delta,A)$
is in the main subpartition of $\hat \gamma$.

Finally, in the case where
there are no $O$ cells, we have $\delta=\hat \gamma$, $|A|=|B|$, and $A \neq B$,
with $B$ as in Lemma~\ref{UniquePair}.
Suppose that $c$ is not in the main subpartition of $\hat \gamma$.
Then, in the $OX$ diagram
of $(\hat \gamma,A)$ there are $X$'s sitting on top of all columns of
the main subpartition.  Since, by the argument about the Bruhat order given
before in the proof,
we must have $\sigma_A(\hat \gamma) \not \subseteq
\sigma_B(\hat \gamma)=\gamma$, and since $\gamma/\hat \gamma$ has a cell over
every column of $\hat \gamma$, this implies that the first row of $\sigma_A(\hat
\gamma)$ is larger than the first row of $\gamma$.
By Remark~\ref{PositionOfBResidues}, all the residues in $B$ are thus
contained in $A$.  Since $|A|=|B|$, this gives the contradiction $A=B$.
\end{proof}

We now describe the involution $\varphi$ (which will only be shown to
be an involution in Proposition~\ref{Propofinal}).
\begin{definition} \label{definitioninvo}
 Let $(\delta,A) \in {\mathcal C}_{\lambda}^{(k)}$.
The involution $\varphi$ is defined according to the two
  following cases:
\begin{enumerate}
\item [{\bf {I}}] Suppose the rightmost changeable cell $c$ in the $OX$
  diagram
of $(\delta,A)$ is an $OX$ cell, and let $i$ be
the row in which it is found.  In this case, the cells with an $OX$ in row $i$ are the only ones with an $O$.  They
correspond to a horizontal ribbon $R$. The involution $\varphi$ then changes the cells where $R$ and its translates are
located into empty cells.   This can be translated in the $(\delta,A)$ language in the following way.   Let $r$ be the
residue of $c$ and let $r'$ be the residue of the leftmost $OX$ cell  in row $i$.  The involution is then $\varphi:
(\delta,A) \mapsto (t_{r',r+1}(\delta), A\backslash\{r\})$.
\item[{\bf {II}}] Suppose the rightmost changeable cell $c$ in the $OX$
  diagram
of $(\delta,A)$ is an empty cell, and let $i$ be
the row in which it is found.  In this case, it can be shown that $c$ has a residue $r$ that does not belong to $A$.
Let $b$ be the leftmost changeable cell in row $i$ whose residue $r'$ is such that $\{r',r'+1,\dots,r-1,r\} \subseteq A
\cup \{r\}$. The involution changes the cells in row $i$ from $b$ to $c$ into $OX$ cells (plus the corresponding
translates). In the $(\delta,A)$ language, this means that $\varphi: (\delta,A) \mapsto (t_{r',r+1}(\delta),
A\cup\{r\})$.
\end{enumerate}
\end{definition}

\begin{example}
Let $k=5$, $\la=(3,3,3,3,3,3,2,2,2,1)$, $\delta
=\hat \gamma = \core (\hat \la)$ and $A=\{ 2,3,5 \}$.   We are in the
case {\bf II} situation since there are no $OX$ cells.
The leftmost changeable cell of $(\delta,A)$ is framed in the
corresponding Ferrers diagram, where the $k$-bounded cells of $\delta$
are given with their
residues.   Then the action of the involution
is:
\begin{equation*}
(\delta, A) = {\tiny \tableau[scY]{~X|4,~X|5,0|0,1|1,2,~X|,,4,5,0|,,5,0,1|,,0,1,2,~X|,,,,,4,5,*0|,,,,,5,0,1,~X,~X}}
\quad \stackrel{\varphi}{\longrightarrow} \quad (t_{5,1}(\delta),A \cup \{ 0 \}) = {\tiny
\tableau[scY]{~X|4,~X|5,0|0,1|1,2,~X|,,4,\overlay X \overlay O \overlay ,\overlay X \overlay O
\overlay|,,5,0,1|,,0,1,2,~X|,,,,,4,\overlay X \overlay O \overlay,*\overlay X \overlay O \overlay|,,,,,5,0,1,~X,~X}}
\end{equation*}
In the resulting $OX$ diagram, the rightmost changeable cell is an $OX$ cell
and we are thus in the case {\bf  I} situation.  The definition of the
involution then takes us back to $(\delta, A)$.

Observe that in $(t_{5,1}(\delta),A \cup \{ 0 \})$, there is a horizontal
ribbon of residues $5$ and $0$. Considering also the
ribbon of residue $1$ below this ribbon, we can form
a vertical $(5,2)$-strip. Take out this vertical $(5,2)$-strip from
$\hat \gamma$, to form $\delta=(7,6,5,4,3,2,2,2,1)$.
Adding the residues $A=\{ 1,2,3,4,5 \}$ leads to a case {\bf I}
situation, where again the rightmost changeable cell has been framed:
\begin{equation*}
(\delta, A) = {\tiny \tableau[scY]{~X|4,~X|5,0|0,1|1,2,~X|,,4,\overlay X \overlay O \overlay,O|,,5,0,\overlay X
\overlay O \overlay|,,0,1,2,~X|,,,,,4,\overlay X \overlay O \overlay,O|,,,,,5,0,
*\overlay X \overlay O \overlay ,~X,~X,~X,~X}} \quad
\stackrel{\varphi}{\longrightarrow} \quad (t_{1,2}(\delta),A \setminus \{ 1 \}) =  {\tiny
\tableau[scY]{~X|4,~X|5,0|0,1|1,2,~X|,,4,\overlay X \overlay O \overlay,O|,,5,0,1|,,0,1,2,~X|,,,,,4,\overlay X \overlay
O \overlay,O|,,,,,5,0,*1,~X,~X,~X,~X}}
\end{equation*}
Again, we see that the action of $\varphi$ on the resulting diagram
takes us back to the initial pair $(\delta,A)$.
\end{example}

\begin{lemma} \label{LemmaCase1r'-1}
In case {\bf {I}}, residue $r'-1$ does not belong to $A$.
\end{lemma}
\begin{proof}
In row $i$ of $\delta$, there is an extremal cell of residue $r'-1$ at the end of its row. Suppose $r'-1 \in A$. Let
$\sigma_A= \sigma_{A'} \sigma_{A''}$ where $r'-1$ belongs to $A'$ and $r'-2$ (if it exists) belongs to $A''$.  In $\sigma_{A''}(\delta)$,
in row $i$, the extremal cell of residue $r'-1$ is still at the end of its row.  Now, when time comes to act with
$\sigma_{r'-1}$, for $\sigma_A$ to increase the number of $k$-bounded hooks by $|A|$, there needs to be an addable corner of
residue $r'-1$ in $\sigma_{A''}(\delta)$.  We will now show that this addable corner is above row $i$, which will lead to
the contradiction that  in $\sigma_{A''}(\delta)$  there is an extremal cell of residue $r'-1$ above row $i$ that is not at
the end of its row (see Proposition~\ref{propo4.1}).

\begin{center}
\begin{picture}(100,90)
\put(-60,0){\includegraphics{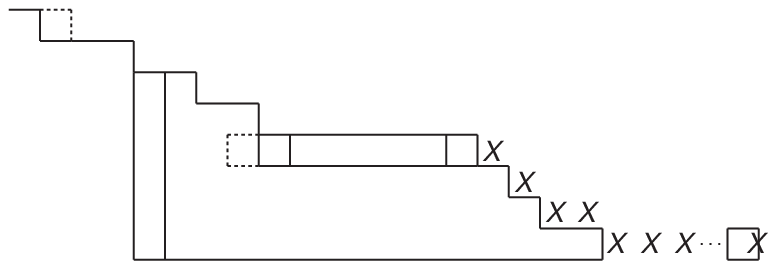}}
\put(69,43){\tiny $r$}
\put(69,33){\footnotesize $c$}
\put(-49,79){\tiny $r'-1$}
\put(15,43){\tiny $r'$}
\put(15,33){\footnotesize $b$}
\put(150,15){\footnotesize $a$}
\put(-36,28.5){\footnotesize $h_1$}
\put(92,33.5){\footnotesize $\leftarrow$ row $i$}
\end{picture}
\end{center}

Since there is no changeable cell to the right of $c$, in the $OX$ diagram associated to $(\delta,A)$, every column to
the right of $c$ ends with an $X$ (and none of them contains an $O$ by
Lemma~\ref{PositionOfChangeableCell}).
Suppose there is an $X$ of residue $r'-1$ below $c$ in the $OX$ diagram of
$(\delta,A)$, and thus sitting on top of a cell of residue $r'$. We have seen that $c$ belongs to the main subpartition
and thus by Remark \ref{rembounded}, this $X$ can only occur in the first row  of the diagram.  But this means that
$X$'s are also found in the first row up until at least residue $r$, or else we find a cell having hook-length equal to
$k+1$. Let the cell with an $X$ of residue $r$ in the first row of the diagram be called $a$.  By a previous comment,
every column between  $c$ and $a$ ends with an $X$, and thus this amounts to
exactly
$k+1-(i-1)=k-i+2$ (recall that $i$ is
the row of $c$) distinct residues of the $X$'s. Recall that $|A|$ is the number of residues to add to $\delta$, thus
$|A| \geq k-i+2$. Now, if the first column of the main subpartition is of height $h_1$, we have $\lambda_1 \leq k-h_1$.
The number of horizontal ribbons removed when going from $\hat \gamma$ to $\delta$ is at most $h_1-i+1$ (one in each row
above $c$ in the main subpartition plus the one in row $i$). Since $|A|=\la_1 +
\textit{the number of horizontal ribbons removed}$,
then $|A| \leq k-h_1 +(h_1-i+1)=k-i+1$. This is a contradiction.
\end{proof}

\begin{lemma} \label{LemmaCase2r'-1}
In case {\bf {II}}, residue $r'-1$ does not belong to $A$.
\end{lemma}
\begin{proof}
If the cell immediately to the left of $b$ exists and is free
from above (therefore changeable), then its residue cannot be in $A$ because otherwise
this would violate the definition of $\varphi$ in case {\bf {II}}.

Now suppose that $r'-1 \in A$, and  let $\sigma_A= \sigma_{A'} \sigma_{A''}$
where $r'-1$ belongs to $A'$ and $r'-2$ (if it exists) belongs to $A''$. Whether there is or not a cell to left of
$b$, in $\sigma_{A''}(\delta)$ the cell above $b$ is an addable corner of residue
$r'-1$ (since the cell immediately to the left of $b$, as we just saw,
cannot be free from above).
Thus $\sigma_A$ adds an $X$ above $b$, which leads to the contradiction that
$b$ is not changeable.
\end{proof}

\begin{lemma} \label{LemmaCase2r}
In case {\bf {II}}, residue $r$ does not belong to $A$.
\end{lemma}
\begin{proof}
Suppose that $r \in A$, and  let $\sigma_A= \sigma_{A'} \sigma_{A''}$ where $r$ belongs to $A'$ and $r-1$ (if it exists) belongs to
$A''$. Note that acting with $\sigma_{A''}$ cannot add an $X$ above $c$, since $c$ is changeable, and cannot add an $X$ to
the right of $c$ since if $r+1 \in A$, then $r+1 \in A'$.  Therefore, when $\sigma_r$ acts, the cell $c$ is a removable
corner. This cannot be if $\sigma_A$ is to increase the number of $k$-bounded hooks by $|A|$.
\end{proof}

\begin{proposition} \label{Propofinal}
  The map $\varphi: \mathcal C_{\lambda}^{(k)}\to \mathcal
  C_{\lambda}^{(k)}, (\delta,A) \mapsto (\delta',A')$ is a well-defined
sign-reversing involution (that is, the cardinalities of $A$ and $A'$ are different
modulo 2) such that $\sigma_A(\delta)=\sigma_{A'}(\delta')$.
\end{proposition}
\begin{proof}

Consider case {\bf {I}}.  In this case, $(\delta',A')=(t_{r',r+1}(\delta),A\backslash \{r\})$. Suppose we have the
$(k,\ell)$-strip $\hat \gamma=\omega^{(1)}\supset \omega^{(2)} \supset \cdots \supset \omega^{(\ell+1)}=\delta$, and
that $\omega^{(i)}/\omega^{(i+1)}$ contains the horizontal ribbon $R$. Thus, from Proposition~\ref{T:scover}, we have
$\omega^{(i)}= t_{r',r+1}(\omega^{(i+1)})$. We will now see that
\begin{equation} \label{Case1VerticalStripProof}
\hat \gamma=\omega^{(1)}\supset  \cdots \supset \omega^{(i)}= t_{r',r+1}(\omega^{(i+1)})\supset
t_{r',r+1}(\omega^{(i+2)})  \supset \cdots \supset t_{r',r+1}(\omega^{(\ell+1)})=t_{r',r+1} (\delta)
\end{equation}
is a vertical $(k,\ell-1)$-strip.  By definition, in the $OX$ diagram associated
to $(\delta,A)$ the cells corresponding to
$R$ are $OX$ cells.  This implies that there are no $O$
cells below $R$, and thus that $R$ does not sit on top of any of the ribbons that occur later in the vertical
$(k,\ell)$-strip.  Furthermore, by Lemma~\ref{PositionOfChangeableCell},
 there are no $O$ cells to the right of $c$, and thus
neither there are horizontal ribbons to the right of $R$.
\begin{center}
\begin{picture}(100,75)
\put(-45,0){\includegraphics{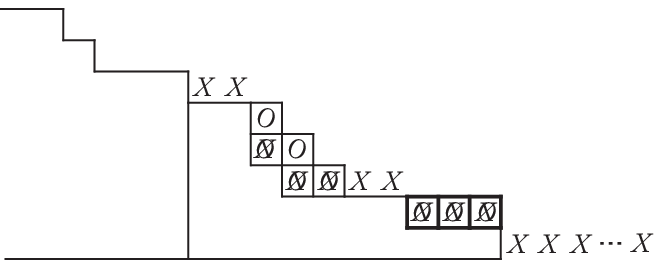}}
\put(85,25){$R$}
\end{picture}
\end{center}
Therefore, by Lemma~\ref{lemmaorder}, the ribbon $R$ (corresponding to $t_{r',r+1}$)  could have been extracted last,
which gives that \eqref{Case1VerticalStripProof} is a vertical $(k,\ell-1)$-strip. Now, since $c$ is filled with an
$OX$, we have $r \in A$. Therefore, $|A\backslash \{r\}|=|A|-1$ and $\varphi$ is a sign-reversing map. Given that $r'-1
\not \in A$ by Lemma \ref{LemmaCase1r'-1}, we can let $\sigma_{A\backslash \{r\}} =\sigma_{D'} \sigma_{D''}$, where
$\sigma_{D''}= \sigma_{r-1} \sigma_{r-2} \cdots \sigma_{r'+1} \sigma_{r'}$. Therefore, using
\begin{equation}
t_{r',r+1} = \sigma_{r'} \sigma_{r'+1} \cdots \sigma_{r-2} \sigma_{r-1} \sigma_r \sigma_{r-1} \sigma_{r-2} \cdots \sigma_{r'+1} \sigma_{r'}
\end{equation}
we find
\begin{equation}
\sigma_{A'}(\delta ')=\sigma_{A\backslash \{r\}} t_{r',r+1} (\delta)= \sigma_{D'} \sigma_{r} \sigma_{r-1} \dots \sigma_{r'} (\delta) = \sigma_{A} (\delta)
\, .
\end{equation}
Therefore case {\bf {I}} is a well defined sign-reversing map
such that $\sigma_A(\delta)=\sigma_{A'}(\delta')$.

Now, we consider case {\bf {II}}.  In this case
$(\delta',A')=(t_{r',r+1}(\delta),A \cup \{r\})$.
From Lemma \ref{LemmaCase2r} we have $|A \cup
\{r\}|=|A|+1$, and thus $\varphi$ is again a sign-reversing
map. By Lemma \ref{LemmaCase2r'-1},  $r'-1 \not \in A$, so we can let $\sigma_{A\cup \{r\}} =\sigma_{D'} \sigma_{D''}$, where
$\sigma_{D''}= \sigma_{r} \sigma_{r-1} \dots \sigma_{r'}$. Therefore, using the same idea as before, we find
\begin{equation} \label{eq5.5}
\sigma_{A'}(\delta ')=\sigma_{A\cup \{r\}} t_{r',r+1} (\delta)= \sigma_{D'} \sigma_{r-1} \dots \sigma_{r'} (\delta) = \sigma_{A} (\delta) \, .
\end{equation}
By definition, we have that $|\kbnd( \sigma_A (\delta) )|=|\kbnd(\delta)|+|A|$. We also have that
\begin{equation}
|\kbnd( \sigma_{A'} (\delta') )|= |\kbnd (\sigma_{A \cup \{ r \}} t_{r',r+1} (\delta))| \leq |\kbnd (t_{r',r+1}
(\delta))| + |A| +1
\end{equation}
since $\sigma_{A \cup \{ r \}}$ can increase the degree of $t_{r',r+1} (\delta)$ by at most the cardinality of $A \cup
\{r \}$.  Using \eqref{eq5.5}, we then find that $|\kbnd( t_{r',r+1} (\delta))| \geq |\kbnd(\delta)| -1$. On the other
hand, from the definition of case {\bf {II}}, we see that applying $\sigma_{r'} \cdots \sigma_{r-1} \sigma_{r}$ on
$\delta$ removes $r-r'+1$ $k$-bounded cells from $\delta$. This gives, using
\begin{equation}
t_{r',r+1} =\sigma_{r} \sigma_{r-1} \cdots \sigma_{r'} \cdots \sigma_{r-1} \sigma_{r} \, ,
\end{equation}
that $|\kbnd( t_{r',r+1} (\delta))| \leq |\kbnd(\delta)| -1$. Therefore, $|\kbnd( t_{r',r+1} (\delta))| =
|\kbnd(\delta)| -1$, and thus, $t_{r',r+1} (\delta) \lessdot \delta$.  By Proposition~\ref{T:scover}, this corresponds
to removing a horizontal ribbon $R$ in row $i$, since applying $t_{r',r+1}$ to $\delta$ removes among other things
 the extremal cells of residues $r', \dots, r$ in row $i$.   Note
that by Lemma~\ref{PositionOfChangeableCell},
the lowest occurrence of $R$ is in the main subpartition of $\hat \gamma$,
and no more $O$'s are found in row $i$ (thus $R$ is the only horizontal
ribbon in its
row).  As a consequence,
if the vertical $(k,\ell)$-strip $\hat \gamma=\omega^{(1)}\supset  \omega^{(2)}
\supset \cdots \supset \omega^{(\ell+1)}=\delta$ is associated to $\delta$, then
the  vertical $(k,\ell+1)$-strip
$\hat \gamma=\omega^{(1)}\supset  \omega^{(2)}
\supset \cdots \supset \omega^{(\ell+1)}=\delta \supset t_{r',r+1} (\delta)$
is associated to $t_{r',r+1} (\delta)$.
Therefore case {\bf {II}} is also a well defined sign-reversing map such that
$\sigma_A(\delta)=\sigma_{A'}(\delta')$.

Finally, observe that the map $\varphi$ is such that the rightmost changeable
cell of $(\delta,A)$ corresponds to the rightmost changeable
cell of $(\delta',A')$.  By construction, and by Lemma~\ref{LemmaCase1r'-1}
which insures that case {\bf {II}} is the inverse of case {\bf {I}},
$\varphi$ is thus an involution.
\end{proof}

\end{document}